\documentclass{amsart}

\usepackage[utf8]{inputenc}
\usepackage[T1]{fontenc}

\usepackage{lmodern}

\usepackage[]{amsmath}
\usepackage{amsthm}
\usepackage{latexsym}
\usepackage{amssymb}
\usepackage{mathrsfs}
\usepackage{exscale}
\usepackage{textcomp}
\usepackage[all,ps,tips,tpic]{xy}
\usepackage{upgreek}
\usepackage{url}
\usepackage{booktabs}
\usepackage[final,pdftex,colorlinks=false,pdfborder={0 0 0}]{hyperref}
\usepackage{aliascnt}

\usepackage{tikz}
\usetikzlibrary{arrows,calc,matrix,patterns}

\newcommand{\Rat}{\mathbb{Q}}
\newcommand{\Nat}{\mathbb{N}}

\newcommand{\Z}{\mathbb{Z}}

\newcommand{\defeq}{\mathrel{\mathop:}=}

\numberwithin{thmcounter}{section}
\newaliascnt{thmauto}{thmcounter}
\newaliascnt{conjauto}{thmcounter}

\newaliascnt{defauto}{thmcounter}

\newaliascnt{exauto}{thmcounter}

\newaliascnt{lemauto}{thmcounter}

\newaliascnt{propauto}{thmcounter}

\newaliascnt{corauto}{thmcounter}

\newaliascnt{remauto}{thmcounter}

\theoremstyle{plain}
\newtheorem{thm}[thmauto]{Theorem}
\newtheorem{conj}[conjauto]{Conjecture}
\newtheorem{ex}[exauto]{Example}

\newtheorem{prop}[propauto]{Proposition}

\newtheorem{thmA}{Theorem}

\theoremstyle{definition}
\newtheorem{definition}[defauto]{Definition}

\theoremstyle{remark}
\newtheorem{rem}[remauto]{Remark}
\newtheorem*{note*}{Note}

\let\originalleft\left
\let\originalright\right
\renewcommand{\left}{\mathopen{}\mathclose\bgroup\originalleft}
\renewcommand{\right}{\aftergroup\egroup\originalright}

\def\polhk#1{\setbox0=\hbox{#1}{\ooalign{\hidewidth
    \lower1.5ex\hbox{`}\hidewidth\crcr\unhbox0}}}

\DeclareMathOperator{\Conf}{Conf}
\DeclareMathOperator{\UConf}{UConf}

\DeclareMathOperator{\id}{id}
\DeclareMathOperator{\Star}{Star}
\DeclareMathOperator{\HH}{H}
\DeclareMathOperator{\map}{map}

\DeclareMathOperator{\TC}{TC}

\DeclareMathOperator{\sh}{sh}

\usepackage[pdftex]{hyperref}
\usepackage{microtype}
\usepackage[draft]{fixme}
\usepackage{soul}

\usepackage{tikz}
\usepackage{pgfplots}
\usepgfplotslibrary{polar}

\newcommand{\particleNr}[2]{\begin{scope}[shift={#1}]\node[fill=white,draw,circle,inner sep=1pt] at
  (0,0) {\small #2};\end{scope}}

\begin{document}
\title[TC of configuration spaces of fully
  articulated and banana graphs]{Topological complexity of configuration spaces of fully
  articulated graphs and banana graphs}

% Information for first author
\author{Daniel Lütgehetmann} \address{University of Aberdeen,
  Aberdeen, United Kingdom} \email{daniel.lutgehetmann@abdn.ac.uk}

% Information for second author
\author{David Recio-Mitter} \address{University of Aberdeen, Aberdeen,
  United Kingdom} \email{david.reciomitter@abdn.ac.uk}


\begin{abstract}
  In this paper we determine the topological
  complexity of configuration spaces of graphs which are not
  necessarily trees, which is a crucial assumption in previous
  results. We do this for two very different classes of graphs:
  fully articulated graphs and banana graphs. %The results give
  %further evidence for a conjecture due to Farber.

  We also complete the computation in the case of trees to include
  configuration spaces with any number of points, extending a proof of
  Farber.

  At the end we show that an unordered configuration space on a graph
  does not always have the same topological complexity as the
  corresponding ordered configuration space (not even when they are
  both connected). Surprisingly, in our counterexamples the topological
  complexity of the unordered configuration space is in fact \emph{smaller} than
  for the ordered one.
\end{abstract}

\maketitle

\section{Introduction}

It is a fundamental problem in industrial robotics to coordinate the
movements of automated guided vehicles (AGVs) along a system of roads
or rails, in such a way that no collisions occur. These situations can
be modeled \cite{Ghrist01} by the configuration space $\Conf_n(G)$ of
$n$ particles on a graph $G$, which is given by
\[ \Conf_n(G) \defeq \left\{ (x_1, \ldots, x_n)\,|\, \text{$x_i\neq
    x_j$ for $i\neq j$} \right\} \subset G^n. \]

Every collision-free movement between two configurations of $n$ points
on the graph $G$ corresponds to a path in the space $\Conf_n(G)$. The
\emph{motion planning problem} consists in finding a function which
assigns to any pair of points a path between them.

Given a topological space $X$ let $p_X:X^I\to X\times X$ denote the
free path fibration on $X$, with projection $p_X(\gamma) =
(\gamma(0),\gamma(1))$. A continuous motion planner on $X$ is
precisely a section of $p_X$. Such a continuous motion planner only
exists in very special cases (in fact, it exists if and only if $X$ is
contractible). Motivated by this, Farber introduced the topological
complexity of a space \cite{Far03}. It is a numerical homotopy
invariant which measures the minimal discontinuity of every motion
planner on this space.

\begin{definition}
  The \emph{topological complexity} of $X$, denoted $\TC(X)$, is
  defined to be the minimal $k$ such that $X\times X$ admits a cover
  by $k+1$ open sets $U_0,U_1,\ldots , U_k$, on each of which there
  exists a local section of $p_X$ (that is, a continuous map
  $s_i:U_i\to X^I$ such that $p_X\circ s_i =
  \mathrm{incl}_i:U_i\hookrightarrow X\times X$).
\end{definition}

Note that here we use the reduced version of $\TC(X)$, which is one
less than the original definition by Farber.

% In the following paragraph we give an overview of the previous
% results.

Let $T$ be a tree and let $|V_{\ge3}|$ denote the number of essential
vertices of $T$ (i.e. vertices with valence at least 3). Farber showed
that $\TC(\Conf_n(T))=2|V_{\ge3}|$ whenever $n\ge2|V_{\ge3}|$; see
\cite{Far05} and also Farber's survey article \cite{Far17}. In
particular, $\TC(\Conf_n(T))$ doesn't depend on $n$ within that range.
% Farber's results are mostly within that ``stable range'', apart from
% his computation of $\TC(\Conf_n(T))$ for $n=1$ and $n=2$ (in those
% two cases $\Conf_n(T)$ is homotopy equivalent to a wedge of
% circles).
Later Scheirer computed the topological complexity for some $n$
outside the aforementioned range \cite{Sch}.
% He found that the value varies with $n$ when $n<2|V_{\ge3}|$.

We complete this picture, extending Farber's argument to compute
$\TC(\Conf_n(T))$ for all $n$, see Theorem \ref{thm:tc-trees}. It should be mentioned that
Scheirer's methods also apply to unordered configuration spaces,
whereas ours do not.

Apart from a few isolated examples, the topological complexity
$\TC(\Conf_n(G))$ has only been computed in the case when $G$ is a
tree. Of course, the requirement that $G$ be a tree is too restrictive
from the point of view of robotics. Indeed, a road system with no
loops in it is bound to be highly inefficient.

A vertex is an \emph{articulation} if removing it makes the graph
disconnected and a connected graph is \emph{fully articulated} if
every essential vertex is an articulation. In
Theorem \ref{thm:tc-separable-graphs} we extend Farber's result to
$\TC(\Conf_n(G))=2|V_{\ge3}|$ for all fully articulated graphs $G$ (of
which trees are a special case) for $n\ge2|V_{\ge3}|$.

In the proof of the results just mentioned we use the
cohomology ring structure of $\Conf_n(G)$; more precisely we use the
\emph{zero-divisor cup-length} (see \autoref{sec:tc}). Our proof
essentially generalizes that of Farber in \cite{Far05}. The key
technical ingredient which enables this generalization is that of
\emph{configuration spaces with sinks}, which were introduced by
Chettih and the first author in \cite{CheLue16}.

The other main result in this paper features a class of graphs which
have no articulations at all, the \emph{banana graphs} $B_k$ (here
$B_k$ denotes the graph with two vertices and $k$ edges connecting
them). In this case we needed a completely different approach. The
starting point is the fact that $\Conf_3(B_4)$ has a particularly nice
homotopy type, namely that of an orientable surface of genus 13
\cite[Proposition 4.3, p. 19]{CheLue16}.  The topological complexity
of surfaces is known to be equal to the zero-divisor cup-length. Using
a Mayer-Vietoris spectral sequence argument involving sinks, similar
to arguments made by the first author in \cite{Luetgehetmann17}, we
are able to gain information about the cohomology ring of
$\Conf_3(B_k)$ for $k\ge4$. This allows us to compute
$\TC(\Conf_n(B_k))$ in all cases except for $k=3$ with $n\ge4$, see
Theorem \ref{thm:tc-banana-graphs}.

In the last section we discuss a conjecture of Farber regarding ordered
configuration spaces of graphs in relation to the results of this paper.
In particular, we show that the conjecture is not true if we replace
``ordered'' by ``unordered'' and give examples where the topological
complexity differs from the ordered to the unordered setting.

The authors would like to thank Mark Grant for very helpful discussions and
suggestions.

\section{Topological complexity}\label{sec:tc}

% For a topological space $X$, let $p_X:X^I\to X\times X$ denote the
% free path fibration on $X$, with projection $p_X(\gamma) =
% (\gamma(0),\gamma(1))$.

Recall the definition from the introduction.

\begin{definition}
  The \emph{topological complexity} of $X$, denoted $\TC(X)$, is
  defined to be the minimal $k$ such that $X\times X$ admits a cover
  by $k+1$ open sets $U_0,U_1,\ldots , U_k$, on each of which there
  exists a local section of $p_X$ (that is, a continuous map
  $s_i:U_i\to X^I$ such that $p_X\circ s_i =
  \mathrm{incl}_i:U_i\hookrightarrow X\times X$).
\end{definition}

In the remainder of this section we state several well-known results
about topological complexity which will be useful later on.

Firstly, the topological complexity $\TC(X)$ is a homotopy invariant
of $X$.

\begin{prop}[\cite{Far03}]\label{lem:retract}
  If $X$ is a homotopy retract of $Y$, then $\TC(X)\le\TC(Y)$.

  Furthermore, if $X$ is homotopy equivalent to $Y$, then
  $\TC(X)=\TC(Y)$.
\end{prop}

The dimension of $X$ gives us a general upper bound for $\TC(X)$.

\begin{prop}[\cite{Far03}]\label{lem:upperbound}
  Let $X$ be a path-connected paracompact space. Then the topological
  complexity of $X$ is bounded above by the covering dimension of the
  product:
  \[
  \TC(X)\le\text{dim}(X\times X).
  \]
  In particular this upper bound holds for all connected CW-complexes
  and in that case $\text{dim}(X\times X)$ is the CW-dimension.
\end{prop}
% \begin{proof}
%   This follows from the upper bound $\TC(X)\le\cat(X\times X)$ given
%   by Farber in \cite{Far03}.
% \end{proof}

Knowing the cohomology ring of a space $X$ can yield lower bounds for
$\TC(X)$ as shown in the following.

\begin{definition}
  Let $X$ be a topological space and $A$ a coefficient ring. A class
  $z\in H^*(X\times X;A)$ is called a \textit{zero-divisor} if the
  pull-back under the diagonal is trivial: $\Delta^*(z)=0$.

  The \emph{zero-divisor cup-length} $zcl_A(X)$ is the length of the
  longest non-trivial product of zero-divisors in $H^*(X\times X;A)$.
\end{definition}

\begin{prop}[\cite{Far03}]\label{lem:lowerbound}
  Let $X$ be a topological space and $A$ a coefficient ring. Then the
  topological complexity of $X$ is bounded below by the zero-divisor
  cup-length:
  \[
  \TC(X)\ge zcl_A(X).
  \]
\end{prop}

Using the previous propositions Farber computed the topological
complexity in the following cases.

\begin{prop}[\cite{Far03}]\label{lem:tc-surface}
  If $\Sigma_g$ is an orientable surface of genus $g\ge2$, then
  \[
  \TC(\Sigma_g)=4.
  \]
\end{prop}

\begin{prop}[\cite{Far04}]\label{lem:tc-graph}
  If $G$ is a connected graph with first Betti number $b_1(G)$, then
  \[
  \TC(G) =
  \begin{cases}
    0 & \text{if $b_1(G)=0$}\\
    1 & \text{if $b_1(G)=1$}\\
    2 & \text{if $b_1(G)\ge2$.}
  \end{cases}
  \]
\end{prop}

\section{Configuration spaces of graphs}

For a topological space $X$ and a finite set $S$ we define the
\emph{configuration space of $X$ with particles labelled by $S$} as
\[ \Conf_S(X) \defeq \left\{f\colon S\to X \text{ injective} \right\}
\subset \map(S, X). \] For $n\in\Nat$ we write $\mathbf{n}\defeq \{1,
2, \ldots, n\}$ and $\Conf_n(X)\defeq \Conf_{\mathbf{n}}(X)$.  This is
usually called the $n$-th ordered configuration space of $X$.  Let $G$
be a finite connected graph (i.e.\ a connected 1-dimensional CW
complex with finitely many cells). We are interested in the
topological complexity of configurations of $n$ ordered particles in
$G$, that is, $\TC(\Conf_n(G))$.

Unless explicitly stated, we assume without loss of generality that
none of the graphs have vertices of valence 2.

A main ingredient in our computations is a modified configuration
space in which particles can collide in some parts of the graph.  This
construction was introduced in \cite{CheLue16} and allows taking
quotients of the underlying space of a configuration space in the
following way.

For a number $n\in\Nat$, a graph $G$ and a subset $W$ of $G$'s
vertices define the following configuration space with sinks:
\begin{equation*}
  \Conf_{n}(G,W) = \left\{(x_1,\ldots,x_n) \in G^n \,|\, \text{for $i\neq j$ either $x_i\neq x_j$ or
      $x_i=x_j\in W$}\right\}.
\end{equation*}
Looking at a collapse map $G\to G/H$ for a subgraph $H\subset G$,
there is now an induced map on configuration spaces if we turn the
image of $H$ under $G\to G/H$ into a sink:
\[ \Conf_n(G) \to \Conf_{n}(G/H, H/H).\]

\subsection{A combinatorial model}
For unordered and ordered configuration spaces of graphs there are
combinatorial models due to Abrams (\cite{Abrams00}), Ghrist
(\cite{Ghrist01}), Świątkowski (\cite{Swiat01}) and the first author
(\cite{Luetgehetmann14}).  In \cite{CheLue16}, a combinatorial model
for configuration spaces \emph{with sinks} inspired by the latter two
models was constructed.  This model is a deformation retract of the
configuration space with the structure of a cube complex.

\begin{definition}[{Cube Complex, see \cite[Definition I.7.32]{bh10}}]
  A cube complex $K$ is the quotient of a disjoint union of cubes
  $X=\bigsqcup_{\lambda\in\Lambda}[0,1]^{k_\lambda}$ by an equivalence
  relation $\sim$ such that the quotient map $p\colon X\to X/\!\!\sim\
  = K$ maps each cube injectively into $K$ and we only identify faces
  of the same dimensions by an isometric homeomorphism.
\end{definition}
\begin{rem}
  The definition above differs slightly from the original definition
  by Bridson and Häfliger, in that it allows two cubes to be
  identified along more than one face.
\end{rem}

\begin{prop}[{\cite[Proposition 2.3, p. 4]{CheLue16}}]\label{prop:combinatorial-model}
  Let $G$ be a finite graph, $W$ a subset of the vertices and
  $n\in\Nat$.  Then $\Conf_{n}(G,W)$ deformation retracts to a finite
  cube complex of dimension $\min\{n, |V_{\ge2}| + |E_W| \}$, where
  $V_{\ge 2}$ is the set of non-sink vertices of $G$ of valence at
  least two and $E_W$ is the set of edges incident to two sinks.
\end{prop}

The basic idea of the combinatorial model is to keep all particles on
any single edge equidistant at all times.  Moving one of the outmost
particles from an edge to an empty essential non-sink vertex is then
given by decreasing the distance of this particle from the vertex
while simultaneously increasing the distance between the particles on
this edge.  Once the particle reaches the vertex, all remaining
particles on the edge will be equidistant again.

More formally, the 0-cubes of the combinatorial model are all those
configurations where all particles in the interior of each edge cut
the edge into pieces of equal length and no particle is in the
interior of any edge incident to one or two sink
vertices. There are no particles on degree 1 vertices which are not sinks.
A $k$-dimensional cube is given by
choosing such a $0$-cell, $k$ distinct particles sitting on distinct
vertices and for each of those particles an edge incident to the
corresponding vertex.  The $i$-th dimension of the cube $[0,1]^k$ then
corresponds to moving the $i$-th of those $k$ particles from their
position on the vertex onto the edge, where at time zero the particle
is on the vertex and at time 1 it is on the edge.  Remember that if
there are already particles on the edge then they continuously squeeze
together to make room for the new particle (it is also possible that
two particles move onto the same edge from different sides).  In the
case where the particle moves onto an edge whose other terminal vertex
is a sink vertex, the particle moves directly into the sink instead.
Such a choice of $k$ movements determines a $k$-cube if and only if we
can realize the movements independently, namely if no two particles
move towards the same non-sink vertex and no two particles move along
the same edge incident to two sink vertices.  This describes the cube
complex as a subspace of the configuration space.

Each non-sink (essential) vertex can only be involved in one of those
combinatorial movements at the same time, so the dimension of this
cube complex is bounded above by the number of essential vertices plus
the number of edges between sink vertices.

\begin{ex}
  The combinatorial model of $\Conf_2(Y)$ for the graph $Y$ shaped
  like the letter Y is given by a circle with six leaves attached to
  it.  More precisely it consists of six univalent vertices, six
  vertices of valence two, six vertices of valence three and eighteen
  edges, see \autoref{fig:conf_2_Y}.  In particular, we have
  $H_1(\Conf_2(Y)) \cong \Z$.
  Remember that particles only move towards vertices of valence at
  least 2 (or sink vertices), so in this example there is only one vertex
  towards which any particle can move.

  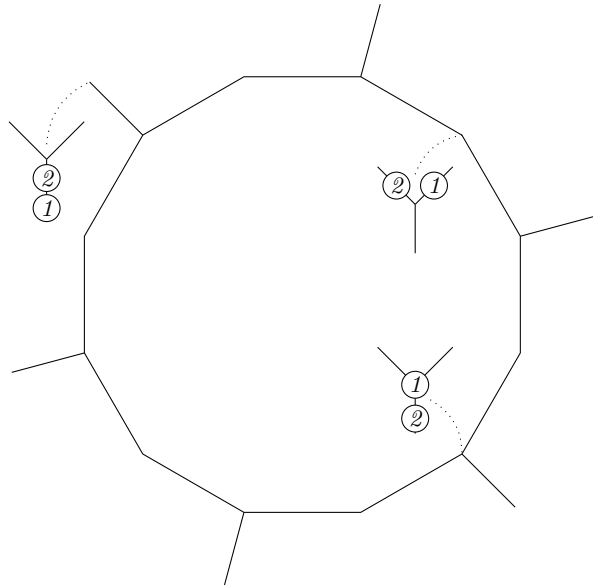
\begin{figure}[htpb]
    \centering
    \begin{tikzpicture}
      \newcommand{\Graph}{ \draw (0, 0) -- (0, -1.3); \draw (0, 0) --
        (1, 1); \draw (0, 0) -- (-1, 1); }

      \draw[-] ($ (0, 0) + (15: 3) $) -- ($ (0, 0) + (45: 3) $) -- ($
      (0, 0) + (75: 3) $) -- ($ (0, 0) + (105: 3) $) -- ($ (0, 0) +
      (135: 3) $) -- ($ (0, 0) + (165: 3) $) -- ($ (0, 0) + (195: 3)
      $) -- ($ (0, 0) + (225: 3) $) -- ($ (0, 0) + (255: 3) $) -- ($
      (0, 0) + (285: 3) $) -- ($ (0, 0) + (315: 3) $) -- ($ (0, 0) +
      (345: 3) $) -- cycle;

      \draw[-] ($ (0, 0) + (15: 3) $) -- ($ (0, 0) + (15: 4) $);
      \draw[-] ($ (0, 0) + (75: 3) $) -- ($ (0, 0) + (75: 4) $);
      \draw[-] ($ (0, 0) + (135: 3) $) -- ($ (0, 0) + (135: 4) $);
      \draw[-] ($ (0, 0) + (195: 3) $) -- ($ (0, 0) + (195: 4) $);
      \draw[-] ($ (0, 0) + (255: 3) $) -- ($ (0, 0) + (255: 4) $);
      \draw[-] ($ (0, 0) + (315: 3) $) -- ($ (0, 0) + (315: 4) $);

      \begin{scope}[shift={(1.5, -1.2)}, scale=0.5]
        \Graph
        \particleNr{(0, 0)}{1};
        \particleNr{(0, -0.9)}{2};
      \end{scope}

      \begin{scope}[shift={(1.5, 1.2)}, scale=0.5]
        \Graph
        \particleNr{(0.5, 0.5)}{1};
        \particleNr{(-0.5, 0.5)}{2};
      \end{scope}

      \begin{scope}[shift={(-3.4, 1.8)}, scale=0.5]
        \Graph
        \particleNr{(0, -1.3)}{1};
        \particleNr{(0, -.5)}{2};
      \end{scope}

      \draw[dotted] (-3.4, 2) to[bend left] ($ (0, 0) + (135: 4) $);
      \draw[dotted] (1.7, -1.4) to[bend left] ($ (0, 0) + (315: 3) $);
      \draw[dotted] (1.5, 1.6) to[bend left] ($ (0, 0) + (45: 3) $);
    \end{tikzpicture}

    \caption{The combinatorial model of $\Conf_2(Y)$. Each edge
      corresponds to the movement of a single particle from the
      essential vertex onto one of the three edges. Moving along the
      embedded circle the two particles move alternatingly onto the
      edge that is not occupied by the other particle.}
    \label{fig:conf_2_Y}
  \end{figure}
\end{ex}

For the interval with two sinks we first consider the cases of few
particles. In the case with only one particle there is no
1-dimensional class and with two particles there is precisely one
1-class: both particles sit on the first sink, particle 1 moves to the
second sink, particle 2 follows, particle 1 returns to the first sink
and finally also particle 2 moves back to the first sink, see
\autoref{fig:comb-model-interval}.

\begin{ex} \label{ex:sinks-interval}
  In \autoref{fig:comb-model-interval} the combinatorial model of
  $\Conf_2([0,1], \{0,1\})$ is visualized.  It is a circle constructed
  by gluing four edges together.

  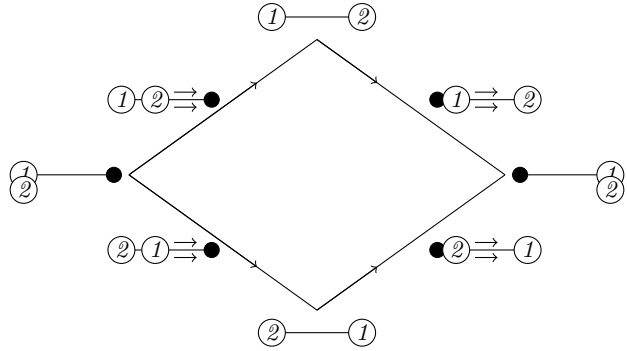
\begin{figure}[ht]
    \begin{center}
      \begin{tikzpicture}
        \newcommand{\Graph}{ \draw (0, 0) -- (1.2, 0); \filldraw (0,0)
          circle (.1cm); \filldraw (1.2,0) circle (.1cm); } \draw
        (-2.5, 0) -- (0, 1.8) -- (2.5, 0); \draw (-2.5, 0) -- (0,
        -1.8) -- (2.5, 0); \draw[->] (-2.5, 0) -- (-0.8, 1.224);
        \draw[->] (0, 1.8) -- (0.8, 1.224);

        \draw[->] (-2.5, 0) -- (-0.8, -1.224); \draw[->] (0, -1.8) --
        (0.8, -1.224);

        % Leftmost
        \begin{scope}[shift={(-3.9, 0)}]
          \Graph
          \particleNr{(0, 0)}{1}
          \particleNr{(0, -.2)}{2}
        \end{scope}

        % Second top
        \begin{scope}[shift={(-2.6, 1.0)}]
          \Graph
          \particleNr{(0, 0)}{1}
          \particleNr{(.45, 0)}{2}
          \draw[->] (0.7, 0.1) -- (1.0, 0.1); \draw[->] (0.7, -0.1) --
          (1.0, -0.1);
        \end{scope}

        % Second bottom
        \begin{scope}[shift={(-2.6, -1.0)}]
          \Graph
          \particleNr{(0, 0)}{2}
          \particleNr{(.45, 0)}{1}
          \draw[->] (0.7, 0.1) -- (1.0, 0.1); \draw[->] (0.7, -0.1) --
          (1.0, -0.1);
        \end{scope}

        % Middle top
        \begin{scope}[shift={(-.6, 2.1)}]
          \Graph
          \particleNr{(0, 0)}{1}
          \particleNr{(1.2, 0)}{2}
        \end{scope}

        % Middle bottom
        \begin{scope}[shift={(-.6, -2.1)}]
          \Graph
          \particleNr{(0, 0)}{2}
          \particleNr{(1.2, 0)}{1}
        \end{scope}

        % Second to last top
        \begin{scope}[shift={(1.6, 1.0)}]
          \Graph
          \particleNr{(.25, 0)}{1}
          \particleNr{(1.2, 0)}{2}
          \draw[->] (0.5, 0.1) -- (0.8, 0.1); \draw[->] (0.5, -0.1) --
          (0.8, -0.1);
        \end{scope}

        % Second to last bottom
        \begin{scope}[shift={(1.6, -1.0)}]
          \Graph
          \particleNr{(.25, 0)}{2}
          \particleNr{(1.2, 0)}{1}
          \draw[->] (0.5, 0.1) -- (0.8, 0.1); \draw[->] (0.5, -0.1) --
          (0.8, -0.1);
        \end{scope}

        % Rightmost
        \begin{scope}[shift={(2.7, 0)}]
          \Graph
          \particleNr{(1.2, 0)}{1}
          \particleNr{(1.2, -.2)}{2}
        \end{scope}
      \end{tikzpicture}
    \end{center}
    \caption{The combinatorial model of $\Conf_2([0,1], \{0, 1\})$
      consists of four edges.}\label{fig:comb-model-interval}
    %\label{fig:combinatorial-model-interval-sinks}
  \end{figure}
\end{ex}

\section{Fully articulated graphs}

\begin{definition}
  A vertex of a connected graph is called an \emph{articulation} if
  removing it makes the graph disconnected.  A connected graph is
  \emph{fully articulated} if all essential vertices are
  articulations.
\end{definition}

Every tree is fully articulated, but the class of fully articulated
graphs is much larger than just trees. For instance, every graph can
be turned into a fully articulated graph by adding a leaf or a loop at
every essential vertex.

\begin{figure}[h]
  \centering
  \includegraphics[scale=0.5]{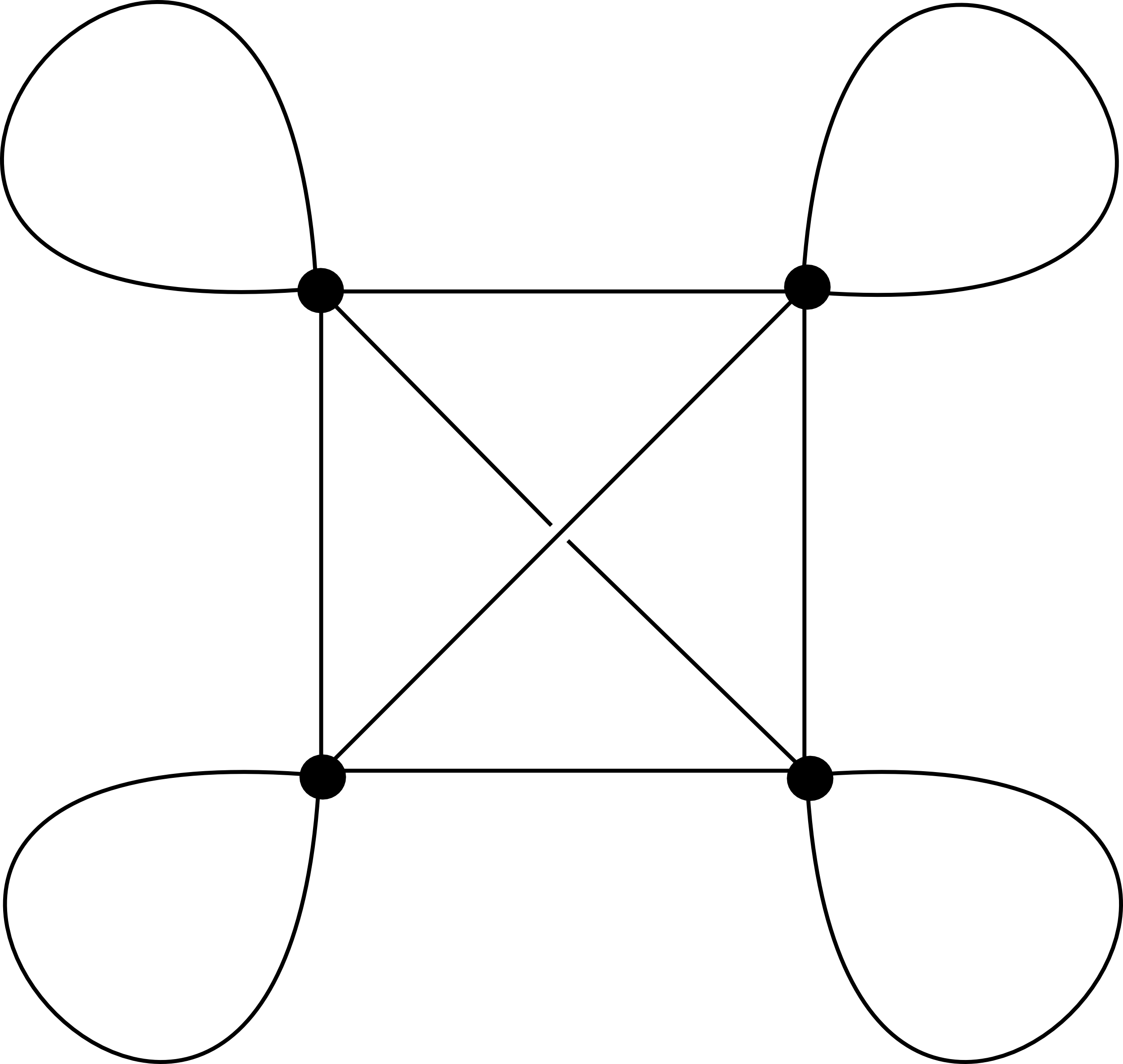}
  \caption{A fully articulated graph.}
  \label{fig:articulated}
\end{figure}

The number of articulations of a graph yields a lower bound for the
topological complexity of configuration spaces on that graph:

\begin{thm}\label{thm:tc-highly-articulated}
  Let $G$ be a graph with $|V_{\ge3}|$ essential vertices, $m\ge2$ of which are
  articulations, and let $n\ge4$. Then
  \[
  2 \min\{\lfloor n/2\rfloor,m\}\le\TC(\Conf_n(G)) \le 2\min\{n,|V_{\ge3}|\}.
  \]
\end{thm}

The following two theorems will follow from Theorem \ref{thm:tc-highly-articulated}.

\begin{thmA}\label{thm:tc-separable-graphs}
  Let $G$ be a fully articulated graph with at least one essential
  vertex. Further assume that $G$ is not homeomorphic to the letter Y. Then the
  topological complexity of $\Conf_n(G)$ for $n\ge 2|V_{\ge3}|$ is given by
  \[
  \TC(\Conf_n(G)) = 2|V_{\ge3}|.
  \]
  If $G$ is homeomorphic to the letter Y, then
  \[
  \TC(\Conf_n(G)) =
  \begin{cases}
    1  & \text{if $n=2$}\\
    2 & \text{if $n\ge3$}.
  \end{cases}
  \]
%
% Furthermore, for $G$ any graph with $m$ articulations we get.
%  \[
%  2 \min\{\lfloor n/2\rfloor,m\}\le\TC(\Conf_n(G)) \le 2\min\{\lfloor n\rfloor,|V_{\ge3}|\}
%  \]
\end{thmA}

\setcounter{thmA}{19}
\begin{thmA}\label{thm:tc-trees}
  Let $T$ be a tree which is not homeomorphic to an interval or to the letter Y. Then
  the topological complexity of $\Conf_n(T)$ is given by
  \[
  \TC(\Conf_n(T)) = 2 \min\{\lfloor n/2\rfloor,|V_{\ge3}|\}.
  \]
\end{thmA}

%Actually the proof of the theorem yields the following more general
%result, which narrows down the topological complexity of highly
%articulated graphs.

The proof of Theorem \ref{thm:tc-highly-articulated} generalizes Farber's
argument from \cite{Far05} using configuration spaces with sinks.
Before we can begin the proof of the three main theorems above we
need some auxiliary propositions.

\begin{prop}[{\cite[Theorem 2.2.4, p. 17]{Chettih16}}]\label{prop:conf-tree-dim}
  Let $G$ be a tree. Then $\Conf_n(G)$ has homotopy dimension
  $\min\{\lfloor n/2\rfloor,|V_{\ge3}|\}$.
\end{prop}

Notice that the above dimension estimate is not stated in this way in
\cite{Chettih16}, but instead a description of the critical cells of a
discrete Morse flow on $\Conf_n(G)$ for $G$ a tree is given.  A
critical cell of dimension $k$ in this description needs at least $2k$
particles and $k$ essential vertices, proving the statement.

% \begin{prop}\label{prop:y-graph-circle}
%   Let $G$ be a graph homeomorphic to the letter Y, or \emph{Y
%   graph}. Then $\Conf_2(G)$ is homotopy equivalent to the circle.
% \end{prop}

\begin{prop}[{\cite{Ghrist01},
\cite[Proposition 3.5, p.  36]{Luetgehetmann14}}]\label{prop:rank-configuration}
Let $Y^l_k$ be a graph which is constructed by attaching $k$ leaves
and $l$ loops to a single vertex, such that $k+2l\ge3$. Then
$\Conf_n(Y^l_k)$ is homotopy equivalent to a graph with first Betti numbers
\[
1 + \frac{(n + k +l -2)!}{(k+l-1)!} (n(k+2l-2) - (k+l)+1).
\]
\end{prop}

\begin{prop}\label{prop:homology-injection}
  Let $G$ be a connected graph with at least one articulation $v$ and
  let $G_v$ be a Y-graph embedded around that vertex such that
  $G_v-\{v\}$ meets at least two connected components of
  $G-\{v\}$. Then the induced map
  \[
  H_1(\Conf_2(G_v);\Z)\to H_1(\Conf_2(G);\Z)
  \]
  is injective.  More precisely, there is a quotient graph with sinks
  $(G,\emptyset)\twoheadrightarrow (\overline{G}_v, W_v)$ such that
  the induced composition
  \[
  H_1(\Conf_2(G_v);\Z) \to H_1(\Conf_2(G);\Z) \to
  H_1(\Conf_2(\overline{G}_v,W_v);\Z)
  \]
  is injective.  Furthermore, the image has rank 1 because
  $\Conf_2(G_v)$ is homotopy equivalent to the circle.
\end{prop}

\begin{proof}
  Denote by $E_v$ the set of all edges not incident to $v$, then we
  define $\overline{G}_v$ to be the graph $G$ with all those edges
  collapsed to points.  Additionally, we add an artificial 2-valent
  sink vertex in the middle of each edge forming a self-loop at $v$.
  Let $W_v$ be the set of all vertices of $\overline{G}_v$ except for
  $v$.  By the total separation assumption the set $W_v$ contains at
  least two elements, and by the choice of embedding of the graph
  there are at least two edges of the image of $G_v$ in
  $\overline{G}_v$ pointing towards distinct sink vertices.

  The graph $\overline{G}_v$ has exactly one non-sink vertex and does
  not have any edge incident to two sink vertices, so by
  \autoref{prop:combinatorial-model} its combinatorial model is a
  graph.  We will now see that a representative of the generator of
  $H_1(\Conf_2(G_v))$ is mapped to a non-trivial cycle in
  $\Conf_2(\overline{G}_v, W_v)$, therefore representing a non-trivial
  homology class.  The representative of the single class in
  $H_1(\Conf_2(G_v))$ is shown in
  \autoref{fig:generator-two-particles-y-graph}.  We now map this
  cycle into the combinatorial model of $\Conf_2(\overline{G}_v, W_v)$
  and then by the total separability assumption there are two
  possibilities: either the three edges point to three distinct sink
  vertices of the quotient or they point to only two different sink
  vertices.  In the first case it is easy to check that the 12
  distinct edges of the cycle map injectively into the combinatorial
  model of the quotient, so it remains to show that the image is a
  non-trivial linear combination of edges also for the second case.

  The image of the cycle in this case is represented as indicated in
  \autoref{fig:generator-two-particles-quotient}.  Cancelling all
  possible edges in the linear combination we get a sum of four circle
  classes: each of the two particles moves along an embedded circle in
  the two possible directions with the other particle sitting on each
  of the two sink vertices.  All edges in these four summands of the
  cycle are distinct edges of the combinatorial model, showing that
  the linear combination is non-trivial.  The lack of 2-cells shows
  that this cycle represents a non-trivial homology class. \qedhere

  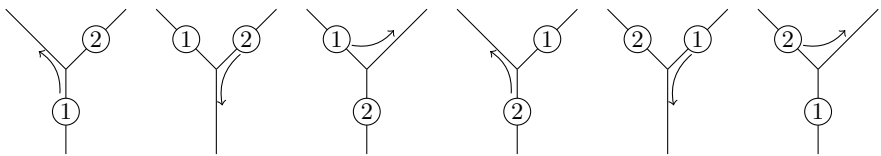
\begin{figure}[ht]
    \begin{center}
      \begin{tikzpicture}[scale=0.8]
        \newcommand{\Graph}{ \draw (0, 0) -- (0, -1.41); \draw (0, 0)
          -- (1, 1); \draw (0, 0) -- (-1, 1); }

        \begin{scope}[shift={(0, 0)}]
          \Graph
          \particleNr{(0, -0.705)}{1}
          \particleNr{(0.5, 0.5)}{2}
          \draw[->, bend right] (-0.1, -0.4) to (-0.45, 0.3);
        \end{scope}

        \begin{scope}[shift={(2.5, 0)}]
          \Graph
          \particleNr{(-0.5, 0.5)}{1}
          \particleNr{(0.5, 0.5)}{2}
          \draw[<-, bend left] (0.1, -0.6) to (0.4, 0.25);
        \end{scope}

        \begin{scope}[shift={(5, 0)}]
          \Graph
          \particleNr{(-0.5, 0.5)}{1}
          \particleNr{(0, -0.705)}{2}
          \draw[->, bend right] (-0.25, 0.4) to (0.45, 0.65);
        \end{scope}

        \begin{scope}[shift={(7.5, 0)}]
          \Graph
          \particleNr{(0.5, 0.5)}{1}
          \particleNr{(0, -0.705)}{2}
          \draw[->, bend right] (-0.1, -0.4) to (-0.45, 0.3);
        \end{scope}

        \begin{scope}[shift={(10, 0)}]
          \Graph
          \particleNr{(0.5, 0.5)}{1}
          \particleNr{(-0.5, 0.5)}{2}
          \draw[<-, bend left] (0.1, -0.6) to (0.4, 0.25);
        \end{scope}

        \begin{scope}[shift={(12.5, 0)}]
          \Graph
          \particleNr{(0, -0.705)}{1}
          \particleNr{(-0.5, 0.5)}{2}
          \draw[->, bend right] (-0.25, 0.4) to (0.45, 0.65);
        \end{scope}
      \end{tikzpicture}
    \end{center}
    \caption{A generator of the first homology of the configuration
      space of two particles in the $Y$ graph. The arrows indicate two
      edges of the combinatorial model, so the whole cycle is a sum of
      12 edges.}
    \label{fig:generator-two-particles-y-graph}
  \end{figure}

  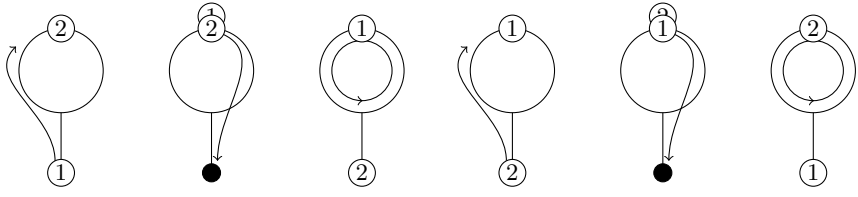
\begin{figure}[ht]
    \begin{center}
      \begin{tikzpicture}[scale=0.8]
        \newcommand{\Graph}{ \draw (0, 0.7) circle (0.7); \draw (0, 0)
          -- (0, -1); \filldraw (0,1.4) circle (.15cm); \filldraw
          (0,-1) circle (.15cm); }

        \begin{scope}[shift={(0, 0)}]
          \Graph
          \particleNr{(0, -1)}{1}
          \particleNr{(0, 1.4)}{2}
          \draw[->, out=90, in=225] (-0.1, -0.8) to (-0.75, 1.1);
        \end{scope}

        \begin{scope}[shift={(2.5, 0)}]
          \Graph
          \particleNr{(0, 1.6)}{1}
          \particleNr{(0, 1.4)}{2}
          \draw[<-, out=90, in=-25] (0.1, -0.8) to (0.2, 1.3);
        \end{scope}

        \begin{scope}[shift={(5, 0)}]
          \Graph \draw (0, 0.7) circle (0.5); \draw[->] (0, 0.2) to
          (0.01, 0.2);
          \particleNr{(0, 1.4)}{1}
          \particleNr{(0, -1)}{2}
        \end{scope}

        \begin{scope}[shift={(7.5, 0)}]
          \Graph \draw[->, out=90, in=225] (-0.1, -0.8) to (-0.75,
          1.1);
          \particleNr{(0, 1.4)}{1}
          \particleNr{(0, -1)}{2}
        \end{scope}

        \begin{scope}[shift={(10, 0)}]
          \Graph
          \particleNr{(0, 1.6)}{2}
          \particleNr{(0, 1.4)}{1}
          \draw[<-, out=90, in=-25] (0.1, -0.8) to (0.2, 1.3);
        \end{scope}

        \begin{scope}[shift={(12.5, 0)}]
          \Graph \draw (0, 0.7) circle (0.5); \draw[->] (0, 0.2) to
          (0.01, 0.2);
          \particleNr{(0, 1.4)}{2}
          \particleNr{(0, -1)}{1}
        \end{scope}
      \end{tikzpicture}
    \end{center}
    \caption{The image of the generator of $H_1(\Conf_2(G_v))$ under
      the collapse map to $H_1(\Conf_2(\overline{G}_v, W_v))$. The
      only cancellation of edges happens amongst the first and fifth
      as well as the second and fourth pictures.}
    \label{fig:generator-two-particles-quotient}
  \end{figure}
\end{proof}

\begin{proof}[Proof of Theorem \ref{thm:tc-highly-articulated}]

  The upper bound follows from
  \autoref{prop:combinatorial-model} (we may assume there are no
  vertices of valance 2) and \autoref{lem:upperbound}.

  It remains to show the lower bounds. Let $d=\min\{\lfloor
  n/2\rfloor,m\}\ge2$. Choose $d$ Y-graphs $G_i$ embedded as
  subgraphs in $G$ as in \autoref{prop:homology-injection}, each
  around a different articulation.  Denote by $(\overline{G}_i,
  W_i)$ the corresponding quotient graphs.  We may assume that the
  $G_i$ are disjoint from each other by sufficiently subdividing $G$.

  % There exist $\lfloor n/2 \rfloor$ pairs of points of the form
  % $(i,i+1)$, indexed by $1\le i\le\lfloor n/2\rfloor$. Further we

  Denote by $\Phi_{i,j}\colon\Conf_n(G)\to\Conf_2(\overline{G}_i,
  W_i)$ the map which sends the configuration $(x_1,\ldots,x_n)$ to
  the image of $(x_{2j-1},x_{2j})$ in the quotient graph, for $1 \le j
  \le d$.

  Since the $G_i$ are disjoint by construction, for every permutation
  $\sigma\in\Sigma_d $ there is an embedding $\psi_\sigma$ which makes
  the following diagrams commute for every $1\le i\le d$.
  \[
  \xymatrix{ \prod_{j=1}^d\Conf_2(G_j)\ar[d]^{\pi_i}
    \ar[0,2]^-{\psi_\sigma}
    &&\Conf_n(G)\ar[d]^{\Phi_{i,\sigma(i)}}\\
    \Conf_2(G_{i}) \ar[r] &\Conf_2(G) \ar[r] &\Conf_2(\overline{G}_i,
    W_i)}
  \]
  Notice that there are choices involved, but the rest of this proof is
  independent of those choices.

  By \autoref{prop:homology-injection}, $H_1(\Conf_2(G_i);\Z/2)$
  embeds as a copy of $\Z/2$ inside the module
  $H_1(\Conf_2(\overline{G}_i, W_i);\Z/2)$.  Let $x_i\in
  H^1(\Conf_2(\overline{G}_i, W_i);\Z/2)$ denote the cohomology class
  dual to the generator of this copy of $\Z/2$.  Consider the classes
  $u_{i,j}=\Phi_{i,j}^*(x_i)\in H^1(\Conf_n(G))$.

  Write $\overline u_{j,k}=u_{j,k}\times1-1\times u_{j,k}$.  All
  elements of this form are zero-divisors.  By
  \autoref{lem:lowerbound} the length of each non-zero product of
  zero-divisors is a lower bound for the topological complexity.  In
  the following we will show that the product
  $\prod_{i=1}^d\overline u_{i,i}\cdot\prod_{i=1}^d\overline
  u_{i,i+1}$ of length $2d$ is non-zero (the indices in $u_{i,i+1}$
  are considered modulo $d$).  % We will deal with the case $d=1$ at the
  % end.

  Let $z_\sigma\in H_d(\Conf_n(G))$ be the image of the generator of
  \[H_d\left(\prod_{j=1}^d\Conf_2(G_j);\Z/2\right)\cong \Z/2\] under
  $(\psi_\sigma)_*$.
  % Observe that by the commutativity of the diagram we have
  % $\langle\prod_{i=1}^du_{i,\sigma_1(i)},z_{\sigma_2}\rangle\ne0$ if
  % and only if $\sigma_1=\sigma_2$.
  Denote by $z$ and $z_{\sh}$ the tori $z_\sigma$ for $\sigma$ the
  identity and the shift $i\mapsto i+1$, respectively.  We will now
  prove that the product of zero-divisors is non-trivial by showing
  that
  \[
  \left\langle\prod_{i=1}^d\overline
    u_{i,i}\cdot\prod_{i=1}^d\overline u_{i,i+1}, z\times z_{\sh}
  \right\rangle \neq 0 \in \Z/2.
  \]
  To compute this product, we have to evaluate products of the form
  $\langle\prod_{i=1}^d u_{j_i,k_i},z\rangle$ for sets of
  \emph{distinct} pairs $(j_i, k_i)$.  By definition, this can be
  computed by evaluating the product of the corresponding
  \[ x_{j_1}\cdots x_{j_d} \in H^d \left(\prod_{i=1}^d
    \Conf_2(\overline{G}_{j_i}, W_{j_i})\right) \] on the image of the
  non-trivial element of $\Z/2$ under the map
  \[
  \Z/2\cong H_d\left( \prod_{i=1}^d \Conf_2(G_i) \right)
  \xrightarrow{\psi_{\id}} H_d(\Conf_n(G)) \xrightarrow{\prod
    \Phi_{j_i,k_i}} H_d\left(\prod_{i=1}^d \Conf_2(\overline{G}_{j_i},
    W_{j_i})\right).
  \]
  This map is the tensor product of maps
  \[
  \eta_{j_i,k_i}\colon \Z/2\cong H_d\left( \prod_{i=1}^d \Conf_2(G_i)
  \right)\to H_1(\Conf_2(\overline{G}_{j_i}, W_{j_i})),
  \]
  so we need to evaluate $x_{j_i}$ on the image of the map above for
  each $i$ and multiply the
  results. % The non-trivial class in the image of $\eta_{j_i,k_i}$
  The image of the non-trivial class under $\eta_{j_i,k_i}$ has a
  representative that only meets the vertex of $\overline{G}_{j_i}$
  that is the image of the essential vertex of $G_{k_i}\subset G$.
  Every class in $H_1(\Conf_2(\overline{G}_{j_i}, W_{j_i}))$ not
  meeting the non-sink vertex is trivial, so the map $\eta_{j_i,k_i}$
  can only be non-trivial if $k_i=j_i$.  Since
  % $u_{i,i}u_{i,i}=0$ for dimension reasons and
  by definition
  \[
  \langle u_{1,1}\cdots u_{d,d},z\rangle = 1,
  \]
  this means that $\langle\prod_{i=1}^d u_{j_i,k_i},z\rangle$ is
  non-zero if and only if we have (up to permutation) that $j_i = k_i
  = i$ for all $i$.

  Repeating the analogous reasoning for $z_{\sh}$ we see that
  \[
  \left\langle\prod_{i=1}^d\overline
    u_{i,i}\cdot\prod_{i=1}^d\overline u_{i,i+1} \; , \; z\times
    z_{\sh}\right\rangle
  \]
  has exactly one non-trivial summand, namely
  \[
  \left\langle\prod_{i=1}^du_{i,i}\times\prod_{i=1}^du_{i,i+1} \; , \;
    z \times z_{\sh}\right\rangle,
  \]
  showing that the product is non-trivial.

\end{proof}

\begin{proof}[Proof of Theorem \ref{thm:tc-separable-graphs}]

  The claim follows from Theorem \ref{thm:tc-highly-articulated} in all cases
  except when $|V_{\ge3}|=1$. In this case $G$ is
  homeomorphic to a wedge of $k$ intervals and $l$ circles, such that
  $k+2l\ge3$. By \autoref{prop:rank-configuration}, $\Conf_n(G)$ is
  homotopy equivalent to the circle if $G$ is homeomorphic to the
  letter Y and $n=2$, and homotopy equivalent to a graph with first Betti number at
  least 2 otherwise. By \autoref{lem:tc-graph} this implies that
  $\TC(\Conf_n(G))=1$ if $G$ is homeomorphic to the letter Y and
  $n=2$, and $\TC(\Conf_n(G))=2$ otherwise. \qedhere

\end{proof}

\begin{proof}[Proof of Theorem \ref{thm:tc-trees}]

  By \autoref{prop:conf-tree-dim} the homotopy dimension of $\Conf_n(T)$
  is $\min\{\lfloor n/2\rfloor,|V_{\ge3}|\}$. Together with
  \autoref{lem:upperbound} this yields the upper bound.

  The lower bound for $\min\{\lfloor n/2\rfloor,|V_{\ge3}|\}\ge2$ follows
  immediately from Theorem \ref{thm:tc-highly-articulated} because trees are
  fully articulated.

  Finally it remains to show the claim for
  $\min\{\lfloor n/2\rfloor,|V_{\ge3}|\}=1$. The case $|V_{\ge3}|=1$ is covered by
  Theorem \ref{thm:tc-separable-graphs}. The other
  possibility is that $n\in\{2,3\}$ and
  $|V_{\ge3}|\ge2$.  By \autoref{prop:conf-tree-dim}, the
  configuration space $\Conf_n(T)$ is in this case homotopy equivalent
  to a graph.  By \cite[Theorem A, p. 2]{CheLue16}, we can choose a
  basis of $H_1(\Conf_n(T);\Z/2)$ consisting of star classes and
  $\HH$-classes.  Since there are at least two vertices there is at
  least one star class and one $\HH$-class, which then must be
  linearly independent.  Therefore, the graph has first Betti number at least 2 and
  we get $\TC(\Conf_n(T))=2$.

\end{proof}

\section{Banana graphs}
\begin{definition}
  The banana graph $B_k$ on $k\ge 1$ edges is the graph consisting of
  two vertices connected by $k$ edges.
\end{definition}

\setcounter{thmA}{1}
\begin{thmA}\label{thm:tc-banana-graphs}
  The topological complexity of $\Conf_n(B_k)$ is given by
  \[
  \TC(\Conf_n(B_k)) =
  \begin{cases}
    4 &\text{if $k \ge 4$ and $n\ge 3$,}\\
    2 &\text{if $k\ge 3$ and $n\le 2$ or $k=n=3$.}
  \end{cases}
  \]
\end{thmA}
% \fxnote{$\Conf_4(B_3)$ has cohomological dimension 2. Probably
% $\TC(\Conf_n(B_3))=4$ for $n\ge 4$.}

\begin{rem}
  The case $k\le 2$ is straightforward: By
  \autoref{prop:combinatorial-model} the combinatorial model in these
  cases is 1-dimensional, so it reduces to computing the topological
  complexity of graphs using \autoref{lem:tc-graph}.  For $k\le 2$ and
  $n> k$ the configuration space $\Conf_n(B_k)$ is disconnected, which
  means that it has infinite topological complexity.  The remaining
  cases are
  \begin{align*}
    \TC(\Conf_1(B_1)) &= \TC(B_1) = 0,\\
    \TC(\Conf_1(B_2)) &= \TC(B_2) = \TC(S^1) = 1,\\
    \TC(\Conf_2(B_2)) &= \TC(B_2) = \TC(S^1) = 1.
  \end{align*}
  The last equality holds because the projection $\Conf_2(B_2)\to
  \Conf_1(B_2) = B_2$ has a homotopy inverse given by putting the
  second particle antipodal to the first particle.
\end{rem}

This means that Theorem \ref{thm:tc-banana-graphs} determines the
topological complexity for all pairs $(n,k)$ except for $(n,3)$ with
$n\ge 4$.

The proof has four main ingredients, whose proofs will be the content
of the rest of this section:

\begin{prop}[{\cite[Proposition 4.3, p. 19]{CheLue16}}]\label{prop:conf-3-b4}
  $\Conf_3(B_4)$ is homotopy equivalent to a closed surface of genus
  13.
\end{prop}

\begin{prop}\label{prop:conf-2-banana}
  For any $k\ge3$ the space $\Conf_2(B_k)$ is homotopy equivalent to a
  connected graph of first Betti number at least $k-1$.
\end{prop}

\begin{prop}\label{prop:conf-b4-injects-bk}
  The map $\Conf_3(B_4) \hookrightarrow \Conf_3(B_k)$ for $k\ge 4$
  induces an injection
  \[ H_1(\Conf_3(B_4)) \hookrightarrow H_1(\Conf_3(B_k)). \]
\end{prop}

\begin{prop}\label{prop:adding-particles-banana}
  For each $n\ge m$ and each graph $G$ with at least one essential
  vertex there exists a map
  \[
  \Conf_m(G) \to \Conf_n(G)
  \]
  which composed with the forgetful map
  \[
  \Conf_n(G) \to \Conf_m(G)
  \]
  is homotopic to the identity.  In particular, we have that
  $H_*(\Conf_m(G))$ is a direct summand of $H_*(\Conf_n(G))$.
\end{prop}

% Notice that we do not require the graph to have a univalent vertex.

\begin{proof}[{Proof of Theorem \ref{thm:tc-banana-graphs}}]
  The second case follows from the fact that $\Conf_n(B_k)$ is in that
  case a connected graph with first Betti number at least two, and so has
  topological complexity 2 by \autoref{lem:tc-graph}. For $n=1$ this
  is immediate and for $n=2$ it follows from
  \autoref{prop:conf-2-banana}.

  In remains to show that $\Conf_3(B_3)$ is homotopy equivalent to a
  1-dimensional complex of first Betti number at least two. This can be seen by
  collapsing cells in the combinatorial model as follows. The
  combinatorial model is 2-dimensional because there are two essential
  vertices.  If one of the moving particles in a 2-cell moves onto the
  edge where the bound particle (i.e.\ the particle that does not move freely)
  is, then the 1-cell in the boundary of this 2-cell where this moving particle
  is bound on the edge is not attached to any other 2-cell.  Therefore, we can
  collapse the 2-cell onto the other three 1-cells in its boundary.  After
  collapsing all such cells we can assume that in each 2-cell the bound particle
  is on an edge where none of the other two particles move along.

  For a 2-cell where the two moving particles move on the same edge
  consider the 1-cell in its boundary where one of the particles is
  bound on the edge.  Because we just collapsed all 2-cells where a
  particle is moving on the edge with a bound particle this 1-cell is
  not contained in any other 2-cell either, so we can collapse it and
  assume that in each 2-cell there is always exactly one particle on
  each edge.

  Given such a 2-cell we now consider the 1-cell where one of the
  particles is on the vertex.  There are two additional potential
  2-cells that are incident to that 1-cell, corresponding to the bound
  particle leaving the vertex for one of the remaining two edges.  But
  each of those 2-cells has two particles on a single edge, and since
  we collapsed all such 2-cells there is no other 2-cell attached and
  we can finish the collapse of the combinatorial model onto a
  1-dimensional cube complex as claimed.  The first Betti number of this graph is at
  least 2 because of \autoref{prop:adding-particles-banana}.

  \vspace{1em}

  We will now prove the first case.  The dimension of the
  combinatorial model of $\Conf_n(B_k)$ for $n\ge3$ and $k\ge4$ is 2,
  giving an upper bound on the topological complexity of $4$ via
  \autoref{lem:upperbound}.  We will now show that we can use
  \autoref{lem:lowerbound} to provide a lower bound of 4 as well.

  By \autoref{lem:tc-surface} all closed surfaces of genus at least
  two have topological complexity 4, so the case $\TC(\Conf_3(B_4)) =
  4$ follows from \autoref{prop:conf-3-b4}.  By
  \autoref{prop:adding-particles-banana} and \autoref{lem:retract} it
  suffices to show $\TC(\Conf_3(B_k))=4$ for all $k\ge 4$.

  \vspace{1em}

  Denote by
  \[
  \phi_k\colon \Conf_3(B_4) \hookrightarrow \Conf_3(B_k)
  \]
  the map induced by the inclusion $B_4\hookrightarrow B_k$ for $k\ge
  4$.  By \autoref{prop:conf-b4-injects-bk}, the induced map
  \[ H^1(\phi_k)\colon H^1(\Conf_3(B_k);\Rat) \to
  H^1(\Conf_3(B_4);\Rat) \] is surjective.  In the proof of the lower
  bound of $\TC(\Conf_3(B_4))=\TC(\Sigma_{13})$ in \cite{Far03} Farber
  constructs four classes $u_1, u_2,u_3,u_4\in H^1(\Conf_3(B_4))$ such
  that the cup product of the associated zero divisors
  $\overline{u_1}\cdots\overline{u_4}$ is non-trivial in the
  cohomology of $\Conf_3(B_4)^{\times 2}$.  Now choose preimages
  $v_1,v_2,v_3$ and $v_4$ of these four classes under $H^1(\phi_k)$
  and look at the product of the corresponding zero-divisors
  $\overline{v_1}\cdots\overline{v_4}$ in
  $H^4\left(\Conf_4(B_k)^{\times 2}\right)$.  By construction, this
  element maps to $\overline{u_1}\cdots\overline{u_4}\neq 0$ under the
  ring map
  \[
  H^*(\phi_k^{\times 2})\colon H^*(\Conf_3(B_k)^{\times 2};\Rat) \to
  H^*(\Conf_3(B_4)^{\times 2};\Rat)
  \]
  so it has to be non-trivial as well.  This proves that
  $\TC(\Conf_3(B_k)) = 4$ for all $k\ge 4$.
\end{proof}

We will now prove the stated propositions.

\begin{proof}[{Proof of \autoref{prop:conf-2-banana}}]
  Consider a 2-cube in the combinatorial model of $\Conf_2(B_k)$ from
  \autoref{prop:combinatorial-model}.  This cube has two coordinates,
  corresponding to the movement of the particles 1 and 2 towards the
  two vertices: increasing the horizontal coordinate moves particle
  $1$ from the interior of some edge towards one of the two vertices,
  and increasing the vertical coordinate moves $2$ in the same way
  towards the other vertex.  Restricting to a face of the 2-cube
  corresponds to keeping the corresponding particle on the vertex or
  in the interior of the edge and moving the other particle towards a
  vertex. Exactly one of the four 1-cubes in the boundary of the
  2-cube keeps particle $1$ in the interior of some edge.  This 1-cube
  is \emph{not} incident to any other 2-cube because particle $1$
  cannot move towards the same vertex as particle $2$.  Therefore, we
  can deform the 2-cell by collapsing this 1-cube onto the other three
  1-cubes.

  Repeating this process for all 2-cubes defines a homotopy
  equivalence to a graph. By \autoref{prop:adding-particles-banana}, the first
  Betti number has to be at least $k-1$.
\end{proof}

\begin{proof}[{Proof of \autoref{prop:adding-particles-banana}}]
  We will define a map between the combinatorial models with these
  properties, which by composition with the deformation retraction and
  inclusion determines a map of the ordinary configuration spaces.
  Choose an essential vertex $v$ and three edges $e_1, e_2, e_3$
  incident to $v$.

  For each $k$-cube in the combinatorial model of $\Conf_m(G)$ where
  no particle moves from $e_1$ towards $v$ simply add the $n-m$
  missing particles in ascending order onto $e_1$ between $v$ and all
  other particles on $e_1$.

  Given a cube where one particle $p$ moves from $e_1$ towards $v$ we
  consider the following sequence of movements: move the $n-m$ new
  particles via $v$ onto $e_2$, move $p$ via $v$ onto $e_3$, move the
  $n-m$ particles back onto $e_1$ in the same way and finally move $p$
  onto $v$.  These movements are independent of the movements of the
  other particles in the chosen cube, so we can replace the movement
  of $p$ with this sequence.  This defines a union of cells in the
  combinatorial model of $\Conf_n(G)$, and we define our map to
  stretch the cube we started with onto this strip of cells, see
  \autoref{fig:strip-of-cells}.  It is straightforward to check that
  this gives a continuous map, i.e.\ that the restriction to the
  boundaries of a cell determine the correct map.

  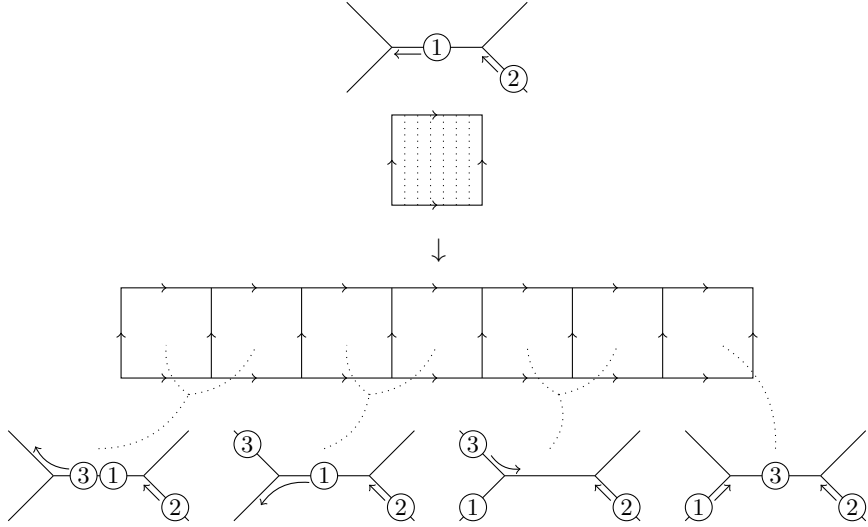
\begin{figure}[htpb]
    \centering
    \begin{tikzpicture}
      \newcommand{\Graph}{ \draw (-1, 0) -- (1, 0); \draw (-2, -1) --
        (-1, 0); \draw (-2, 1) -- (-1, 0); \draw (2, -1) -- (1, 0);
        \draw (2, 1) -- (1, 0); }

      \begin{scope}[shift={(0, 5)}, scale=0.6]
        \Graph
        \particleNr{(0, 0)}{1};
        \particleNr{(1.7, -0.7)}{2}; \draw[->] (-0.35, -0.15) to
        (-0.95, -0.15); \draw[->] (1.35, -0.55) to (1.0, -0.2);
      \end{scope}

      \begin{scope}[shift={(0, 3.5)}, scale=0.6]
        \draw (-1, -1) -- (1, -1) -- (1, 1) -- (-1, 1) -- cycle;
        \draw[->] (-.005, -1) -- (.005, -1); \draw[->] (-.005, 1) --
        (.005, 1); \draw[->] (-1, -.005) -- (-1, .005); \draw[->] (1,
        -.005) -- (1, .005); \draw[dotted] (-0.7143, -1) -- (-0.7143,
        1); \draw[dotted] (-0.4286, -1) -- (-0.4286, 1); \draw[dotted]
        (-0.1429, -1) -- (-0.1429, 1); \draw[dotted] (0.1428, -1) --
        (0.1428, 1); \draw[dotted] (0.4285, -1) -- (0.4285, 1);
        \draw[dotted] (0.7142, -1) -- (0.7142, 1);
      \end{scope}

      \node[rotate=-90] at (0, 2.3) {$\to$};

      \begin{scope}[shift={(-1.8, 1.2)}, scale=0.6]
        \draw (-4, -1) -- (10, -1) -- (10, 1) -- (-4, 1) -- cycle;

        \draw[->] (-3.005, -1) -- (-2.995, -1); \draw[->] (-3.005, 1)
        -- (-2.995, 1); \draw[->] (-1.005, -1) -- (-0.995, -1);
        \draw[->] (-1.005, 1) -- (-0.995, 1); \draw[<-] (1.005, -1) --
        (0.995, -1); \draw[<-] (1.005, 1) -- (0.995, 1); \draw[<-]
        (3.005, -1) -- (2.995, -1); \draw[<-] (3.005, 1) -- (2.995,
        1); \draw[<-] (5.005, -1) -- (4.995, -1); \draw[<-] (5.005, 1)
        -- (4.995, 1); \draw[<-] (7.005, -1) -- (6.995, -1); \draw[<-]
        (7.005, 1) -- (6.995, 1); \draw[<-] (9.005, -1) -- (8.995,
        -1); \draw[<-] (9.005, 1) -- (8.995, 1);

        \draw[->] (-2, -.005) -- (-2, .005); \draw[->] (-4, -.005) --
        (-4, .005); \draw[->] (2, -.005) -- (2, .005); \draw[->] (0,
        -.005) -- (0, .005); \draw[->] (4, -.005) -- (4, .005);
        \draw[->] (6, -.005) -- (6, .005); \draw[->] (8, -.005) -- (8,
        .005); \draw[->] (10, -.005) -- (10, .005); \draw (-2, -1) --
        (-2, 1); \draw (0, -1) -- (0, 1); \draw (2, -1) -- (2, 1);
        \draw (4, -1) -- (4, 1); \draw (6, -1) -- (6, 1); \draw (8,
        -1) -- (8, 1);
      \end{scope}

      \begin{scope}[shift={(0, -0.7)}, scale=0.6]
        \begin{scope}[shift={(-7.5, 0)}]
          \Graph
          \particleNr{(0.33, 0)}{1};
          \particleNr{(-0.33, 0)}{3};
          \particleNr{(1.7, -0.7)}{2}; \draw[->, bend left] (-0.65,
          0.15) to (-1.4, 0.65); \draw[->] (1.35, -0.55) to (1.0,
          -0.2);

          \draw[dotted, bend right] (0, 0.6) to (2, 1.8);
          \draw[dotted, bend left] (2, 1.8) to (1.5, 2.9);
          \draw[dotted, bend right] (2, 1.8) to (3.5, 2.9);
        \end{scope}
        \begin{scope}[shift={(-2.5, 0)}]
          \Graph
          \particleNr{(0, 0)}{1};
          \particleNr{(-1.7, 0.7)}{3};
          \particleNr{(1.7, -0.7)}{2}; \draw[->, bend right] (-0.33,
          -0.15) to (-1.4, -0.65); \draw[->] (1.35, -0.55) to (1.0,
          -0.2);

          \draw[dotted, bend right] (0, 0.6) to (1, 1.8);
          \draw[dotted, bend left] (1, 1.8) to (0.5, 2.9);
          \draw[dotted, bend right] (1, 1.8) to (2.5, 2.9);
        \end{scope}
        \begin{scope}[shift={(2.5, 0)}]
          \Graph
          \particleNr{(-1.7, -0.7)}{1};
          \particleNr{(-1.7, 0.7)}{3};
          \particleNr{(1.7, -0.7)}{2}; \draw[<-, bend left] (-0.65,
          0.15) to (-1.3, 0.45); \draw[->] (1.35, -0.55) to (1.0,
          -0.2);

          \draw[dotted, bend right] (0, 0.6) to (0.2, 1.8);
          \draw[dotted, bend left] (0.2, 1.8) to (-0.5, 2.9);
          \draw[dotted, bend right] (0.2, 1.8) to (1.5, 2.9);
        \end{scope}
        \begin{scope}[shift={(7.5, 0)}]
          \Graph
          \particleNr{(-1.7, -0.7)}{1};
          \particleNr{(0, 0)}{3};
          \particleNr{(1.7, -0.7)}{2}; \draw[->] (-1.35, -0.55) to
          (-1.0, -0.2); \draw[->] (1.35, -0.55) to (1.0, -0.2);

          \draw[dotted, bend right] (0, 0.6) to (-1.2, 2.9);
        \end{scope}
      \end{scope}
    \end{tikzpicture}
    \caption{Replacing a 2-cell by a strip of 2-cells to construct a
      map $\Conf_m(G)\to \Conf_n(G)$ for $m<n$. The vertical direction
      in the cubes corresponds to the movement of particle 2, the
      seven small rectangles above are stretched to the seven cubes
      below.}
    \label{fig:strip-of-cells}
  \end{figure}

  \vspace{1em}

  By construction, the composition with the map forgetting the $n-m$
  new particles gives almost the identity, only the particles moving
  from $e_1$ towards $v$ briefly move onto $e_3$.  Up to homotopy,
  however, the map is the identity.
\end{proof}

\subsection{Proof of \autoref{prop:conf-b4-injects-bk}}
The proof of \autoref{prop:conf-b4-injects-bk} is a combination of
parts of proofs in \cite{Luetgehetmann17}.  For the convenience of the
reader we reproduce the relevant parts here.

\subsection{A Mayer-Vietoris spectral sequence for configuration
  spaces}
\label{sec:mv-spectral-sequence}
In this section, we will use the basic classes in the first homology of
configuration spaces of graphs to prove \autoref{prop:conf-b4-injects-bk}.
These are given by star classes, H-classes and $S^1$-classes, which are elements
in the first homology of configuration spaces of any star graph, the graph that
looks like the letter H and the circle $S^1$, respectively.
$S^1$-classes are given by all particles moving around the circle
simultaneously, H-classes are given by ``exchange movements'' like in
\autoref{ex:sinks-interval} and star classes are shufflings of particles sitting
on the leaves via the central vertex.
For more details on those classes, see for example \cite{Luetgehetmann17}.

We now recall the construction of the Mayer-Vietoris spectral sequence
for configuration spaces discussed in \cite{CheLue16} .
\begin{definition}[Mayer-Vietoris spectral sequence]
  Let $J$ be a countable ordered index set and $\{V_j\}_{j\in J}$ an
  open cover of $X$, then we define the following countable open cover
  $\mathcal{U}(\{V_j\})$ of $\Conf_n(X)$: for each $\phi\colon
  \mathbf{n}\to J$ we define $U_\phi$ to be the set of all those
  configurations where each particle $i$ is in $V_{\phi(i)}$, i.e.
  \[
  U_\phi \defeq \bigcap_{i\in\mathbf{n}} \pi_i^{-1}\left( V_{\phi(i)}
  \right).
  \]
  These sets are open and cover the whole space, so they define a
  spectral sequence
  \[
  E^1_{p,q} = \bigoplus_{\phi_0<\cdots<\phi_p} H_q\left(
    U_{\phi_0}\cap\cdots\cap U_{\phi_p} ;\Rat\right) \Rightarrow
  H_*\left( \Conf_n(X);\Rat \right),
  \]
  where the indexing of the direct sum is over all ordered sets $\phi_0 < \cdots
  < \phi_p$ of maps $\phi_i$ as described above, converging to the homology of
  the whole space. Here, we chose an arbitrary ordering of the maps $\phi_i$,
  for example lexicographic ordering. For an elementary proof of the convergence
  of this spectral sequence, see \cite[Proposition 2.1, p. 25]{Luetgehetmann17}.
\end{definition}
For brevity, we will also write
\[
U_{\phi_0\cdots\phi_p} \defeq U_{\phi_0} \cap \cdots \cap U_{\phi_p}.
\]
The boundary map $d_1$ is given by the alternating sum of the face
maps induced by
\[
U_{\phi_0}\cap\cdots\cap U_{\phi_p} \hookrightarrow
U_{\phi_0}\cap\cdots\cap \widehat{U_{\phi_i}}\cap\cdots\cap U_{\phi_p}
\]
forgetting the $i$-th open set from the intersection.  Of course, this
construction generalizes to configuration spaces with sinks.  From now
on we will suppress the rational coefficients in our notation.

\begin{proof}[{Proof of \autoref{prop:conf-b4-injects-bk}}]
  Endow $B_k$ with a path metric such that each edge has length 1.
  Consider the open cover of $B_k$ given by the open balls $V_1$ and
  $V_2$ of radius $2/3$ around $v_1$ and $v_2$, respectively.  The
  intersection $V_1\cap V_2$ is given by a disjoint union of intervals
  of length $1/3$.  Pulling the particles out of this intersection if
  possible one can see that each intersection $U_{\phi_0\cdots\phi_p}$
  is homotopy equivalent to a disjoint union of spaces of the form
  \begin{align}\label{eq:identification-intersections}
    \Conf_{S_1}(\Star_{v_1}) \times \Conf_{S_\cap}(\sqcup_k I ) \times
    \Conf_{S_2}(\Star_{v_2}),
  \end{align}
  where $\Star_v$ is a small contractible neighbourhood of $v$, $I$ is
  an interval and $S_1\sqcup S_\cap\sqcup S_2 = \{1,2,3\}$.

  Let us now compute the bottom row of the $E^2$-page of the spectral
  sequence $E^*_{\bullet,\bullet}[k]$ associated to the open cover of
  $\Conf_3(B_k)$.  The bottom row of the $E^1$-page is given at
  position $(p, 0)$ by the direct sum of all terms of the form
  $H_0(U_{\phi_0\cdots\phi_p})$.  If we now write the same spectral
  sequence for $\Conf_3(B_k, V(B_k))$, i.e.\ with both vertices turned
  into sinks, then we see by the identification
  \eqref{eq:identification-intersections} that the bottom rows of both
  $E^1$-pages agree (including differentials).  Since in the sink
  case all $U_{\phi_0\cdots\phi_p}$ have contractible path components,
  all higher rows of the $E^1$-page are trivial.  Therefore, the
  bottom row of the $E^2$-pages of both spectral sequences at position
  $p$ is given by $H_p(\Conf_3(B_k, V(B_k)))$.  In particular, this
  gives $E^\infty_{1,0}[k]\cong H_1(\Conf_3(B_k, V(B_k)))$ for both
  spectral sequences.

  By \cite[Theorem D, p. 5]{Luetgehetmann17} the first homology of
  $\Conf_3(B_k, V(B_k))$ is generated by $S^1$-classes and
  $\HH$-classes.  It is straightforward to check that all
  $\HH$-classes in $\Conf_3(B_k, V(B_k))$ are trivial for $k\ge 3$, so
  we can choose a basis of $H_1(\Conf_3(B_k, V(B_k)))$ consisting only
  of individual particles moving along an embedded circle in $B_k$
  with both other particles sitting on one of the sinks.  This means
  that the inclusions
  \[
  \Conf_{\{1\}}(B_k) \sqcup \Conf_{\{2\}}(B_k) \sqcup
  \Conf_{\{3\}}(B_k) \to \Conf_3(B_k, V(B_k))
  \]
  given by putting the remaining particles onto one of the sinks
  induces a surjection in first homology.  Composition with the
  forgetful maps shows that the map is also injective in homology, so
  it in fact induces an $H_1$-isomorphism.  Since
  $H_1(\Conf_1(B_k))=H_1(B_k)$ has $H_1(\Conf_1(B_4))=H_1(B_4)$ as a
  direct summand we get that also $E^\infty_{1,0}[4]=H_1(\Conf_3(B_4,
  V(B_4)))$ is a direct summand of $E^\infty_{1,0}[k]=H_1(\Conf_3(B_k,
  V(B_k)))$.

  \vspace{1em}

  It remains to see that the same splitting is possible for
  $E^\infty_{0,1}[k]$.  The module $E^1_{0,1}[k]$ is by the
  identification \eqref{eq:identification-intersections} given by a
  direct sum of modules of the form
  \[
    H_1(\Conf_{S_1}(\Star_{v_1}[k]) \times \Conf_{S_\cap}(\sqcup_k I)
    \times \Conf_{S_2}(\Star_{v_2}[k])),
  \]
  where $\Star_{v_1}[k]$ and $\Star_{v_2}[k]$ are the stars around the vertices
  $v_1$ and $v_2$.
  It is straightforward to check that the map
  \[
  E^1_{0,1}[k]\supset H_1(\Conf_3(\Star_{v_1}[k]))\oplus
  H_1(\Conf_3(\Star_{v_2}[k])) \to E^2_{0,1}[k]
  \]
  is a surjection and that the images of the two direct summands
  intersect trivially in $E^2_{0,1}[k]$.  Define $Q^{v_1}_2[k]$ as the
  image of the map
  \[
  \bigoplus_{e\in E(\Star_{v_1}[k])}\bigoplus_{|S|=2}
  H_1\left(\Conf_S(\Star_{v_1}[k])\right) \to H_1\left(\Conf_3(\Star_{v_1}[k])\right),
  \]
  where $S\subset\{1,2,3\}$ and the map is induced by adding the third
  particle onto the end of edge $e$ (away from $v_1$).  This submodule
  is everything generated by classes with only two moving particles.
  We now write the module $H_1(\Conf_3(\Star_{v_1}[k]);\Rat)$ as
  $Q^{v_1}_2[k]\oplus Q^{v_1}_3[k]$ for some (arbitrary) choice of
  $Q^{v_1}_3[k]$. Because in $E^1_{1,1}[k]$ there is no term with all three
  particles in one of the two stars, this means that $Q_3^{v_1}[k]$ does not
  intersect the image of the boundary map $d_0$ and therefore is a direct
  summand of $E^1_{0,1}[k]$.  The module $Q^{v_1}_2[k]$ by definition has a
  basis where each basis element has exactly one bound particle (i.e.\ a
  particle that does not move freely).  By adding $d_0$ boundaries to the image
  of such a basis element we can arrange that this bound particle is always on a
  fixed edge $e_0$.  One can now check that the image of $Q^{v_1}_2[k]$ under
  the collapse map to $E^2_{0,1}[k]$ is the same as the image of
  $\overline{Q^{v_1}_2}[k]$, defined as the image of the induced map
  \[
  \bigoplus_{|S|=2} H_1\left(\Conf_S(\Star_{v_1}[k])\right) \to
  H_1\left(\Conf_3(\Star_{v_1}[k])\right),
  \]
  where the third particle is always put onto the edge $e_0$.  There
  are no relations imposed onto the image of $\overline{Q^{v_1}_2}[k]$
  for $v\in\{v_1,v_2\}$ under the quotient map, so we get (by
  repeating the same argument for $v_2$)
  \[
  E^2_{0,1}[k] \cong \overline{Q^{v_1}_2}[k]\oplus Q^{v_1}_3[k]\oplus
  \overline{Q^{v_2}_2}[k]\oplus Q^{v_2}_3[k].
  \]
  Notice that by symmetry we have in fact
  \begin{align}\label{eq:identification-spectral-sequence}
    E^2_{0,1}[k] \cong \overline{Q^{v_1}_2}[k]^{\oplus 2}\oplus
    Q^{v_1}_3[k]^{\oplus 2}.
  \end{align}

  Each $H_1(\Conf_S(\Star_{v}[k]))$ for any finite set $S$ has
  $H_1(\Conf_S(\Star_v[4]))$ as a direct summand: choose a spanning
  tree in the graph $\Conf_S(\Star_v[4])$, then this defines a tree in
  $\Conf_S(\Star_v[k])$ via the inclusion.  Extending this tree to a
  spanning tree shows that the first homology of $\Conf_S(\Star_v[4])$
  is a direct summand of $H_1(\Conf_S(\Star_v[k]))$.  It is
  straightforward to check that this direct sum decomposition respects
  the splitting into $\overline{Q^{v}_2}[k]\oplus Q^v_{3}[k]$ in the
  sense that $\overline{Q^v_2}[4]$ is a direct summand of
  $\overline{Q^v_2}[k]$ and $Q^v_3[4]$ and $Q^v_3[k]$ can be chosen such that
  the former is a direct summand of the latter. This shows that $E^2_{0,1}[4]$
  is a direct summand of $E^2_{0,1}[k]$. Denote the complement by
  $E^2_{0,1}[k,4]$.

  \vspace{1em}

  It remains to show that the boundary map
  \[
  d_2\colon H_2(\Conf_3(B_k, V(B_k))) \to E^2_{0,1}[k] \cong
  E^2_{0,1}[4] \oplus E^2_{0,1}[k,4]
  \]
  preserves this splitting for $k>4$.

  Consider the analogous spectral sequence $\tilde
  E^*_{\bullet,\bullet}[k]$ for the space $\Conf_3(B_k, \{v_2\})$
  instead of $\Conf_3(B_k)$.  In this case we get that the term
  $\tilde E^2_{0,1}[k]$ is given by $\overline{Q^{v_1}_2}[k]\oplus
  Q_3^{v_1}[k]$.  The combinatorial model of $\Conf_3(B_k, \{v_2\})$
  is one dimensional and all 1-cubes have the following form: one
  particle moves from the sink to the other vertex while the other
  particles stay in the sink. Therefore, its first homology is
  generated by individual particles moving along embedded circles with
  the remaining particles fixed on the sink.  Each such class can be
  represented as an element of $\tilde E^\infty_{1,0}[k]$ by looking
  at the intersection of the two open sets where a particle $p$ is in
  the neighbourhood of $v_1$ or $v_2$, respectively, and all other
  particles are in the neighbourhood of $v_2$.  The circle class is
  then represented by the difference of $p$ being in different
  connected components of the disjoint union of intervals given by the
  intersection of the two open subsets of $B_k$.

  This shows that $\tilde E^\infty_{0,1}[k]=\tilde E^3_{0,1}[k]=0$,
  and therefore that
  \[
    \tilde{d}_2\colon H_2(\Conf_3(B_k, V(B_k))) \to \overline{Q^{v_1}_2}[k]
  \oplus Q^{v_1}_3[k]
  \]
  is surjective in this case. Since the second homology of $\Conf_3(B_k,
  \{v_2\})$ is trivial, this map in fact is an isomorphism. This shows with the
  identification \eqref{eq:identification-spectral-sequence} that for the
  spectral sequence for $\Conf_3(B_k)$ we get
  \[
  E^\infty_{0,1} = E^3_{0,1}[k] \cong \overline{Q^{v_1}_2}[k] \oplus
  Q^{v_1}_3[k]
  \]
  because the map $d_2$ is the product of the two isomorphisms $\tilde d_2$ from
  the cases where $v_1$ or $v_2$ is a sink, respectively (we can compare the spectral
  sequences using the induced maps, by naturality). Since we already saw
  that $\overline{Q^{v_1}_2}[4]$ and $Q^{v_1}_3[4]$ are direct
  summands of $\overline{Q^{v_1}_2}[k]$ and $Q^{v_1}_3[k]$,
  respectively, this concludes the proof.
\end{proof}

\section{On a conjecture of Farber and on unordered configuration spaces}

In \cite{Far05} Farber formulated the following
\begin{conj}[Farber]
  Let $G$ be a connected graph with $|V_{\ge3}|\ge2$ and let $n\ge2|V_{\ge3}|$. Then
  \[\TC(\Conf_n(G))=2|V_{\ge3}|.\]
\end{conj}

In the same paper Farber proved that the conjecture holds for
trees. The results in this paper provide further evidence for this
conjecture by showing that it holds for the more general fully
articulated graphs and for most banana graphs.

In Theorem \ref{thm:tc-separable-graphs} we also show that, for $T$ a
tree, \[\TC(\Conf_n(T))= 2 \lfloor n/2 \rfloor\] grows steadily in $n$
while $n<2|V_{\ge3}|$ until it stabilizes at
\[\TC(\Conf_n(T))=2|V_{\ge3}|\] for $n\ge2|V_{\ge3}|$.

This does not generalize to banana graphs: by
Theorem \ref{thm:tc-banana-graphs} we have that \[\TC(\Conf_3(B_k))=4,\] but
$3<2|V_{\ge3}|$.  This raises the problem of understanding the
behaviour of $\TC(\Conf_n(G))$ for small $n$, for a general graph $G$.

Another open question is the relationship between $\TC(\Conf_n(G))$ for ordered configuration spaces and $\TC(\UConf_n(G))$ for unordered configuration spaces. Scheirer showed in \cite{Sch} that they coincide
in many cases and so one might be tempted to conjecture that they are
always equal, provided that $\UConf_n(G)$ is connected. However, this is in fact not the case.

A counterexample is given by
$G=H$ and $n=4$. We know from Theorem \ref{thm:tc-trees} that
$\TC(\Conf_4(H))=4$. However by the work of Connolly and Doig \cite[Prop.~8 and
Prop.~11]{ConDoi14} we see that $\UConf_4(H)$ is the classifying space for
$F_{10}\ast(\Z\times\Z)$ and so $\TC(\UConf_4(H))=3$ by \cite{CohPru08},
where the topological complexity of classifying spaces of right-angled
Artin groups is computed. This shows that Farber's conjecture does not hold true in
the unordered setting, not even for trees.

One can similarly construct infinitely many such counterexamples
by glueing $m$ Y graphs together such that all essential vertices lie on an interval
and taking $n=2m$. It is worth noting that in all these cases
the topological complexity in the unordered setting is in fact
\emph{smaller} than in the ordered one.

% In this section we give some isolated examples of graphs for which
% $\TC(\Conf_n(G))$ and $\TC(\UConf_n(G))$ can be computed.

% In \cite{Far05} Farber observed that the topological complexity of
% $\Conf_2(K_5)$ and $\Conf_2(K_{3,3})$ is 4. This follows from the
% fact that those two spaces are higher genus orientable surfaces (for
% a proof of this fact see Abrams thesis \cite{Abrams00}), together
% with the computation of the topological complexity due to Farber in
% \cite{Far03} (see also \autoref{lem:tc-surface}).

% \fxnote[inline]{On a related note, is it unexpected that
% $\Conf_3(K_5)$ and $\Conf_4(K_{3,3})$ are so low-dimensional and
% what is the maximal dimension of $\Conf_k(K_{3,3})$ for example?}
% \fxnote[inline]{As commented on gitlab, those are not actually
% surfaces. The homological dimension of $\Conf_k(K_{3,3})$ grows
% linearly until it stabilizes at dimension 6. Delete this comment
% when you read it :)}

% Other cases which are surfaces and so the topological complexity is
% 4: , $\UConf_3(K_5)$ and $\UConf_4(K_{3,3})$, $\UConf_2(K_5)$ and
% $\UConf_2(K_{3,3})$.

\bibliographystyle{halpha} \bibliography{tc-conf-graph}
\end{document}